\documentclass[a4paper,12pt]{article}

\usepackage{amsmath,amssymb,amsthm,bm}
\usepackage{tikz}
\usepackage{hyperref}

\usepackage{amsmath,amssymb,amsthm,bm}
\usepackage{refcount}
\AtEndDocument{%
  \ifnum\getpagerefnumber{paper:body-end}>20\relax
    \PackageWarningNoLine{SVPTarget20}{%
Main text exceeds the CiC regular-paper limit of 20 pages}%
  \fi
}
\newtheorem{thm}{Theorem}[section]
\newtheorem{prop}[thm]{Proposition}
\newtheorem{lem}[thm]{Lemma}
\newtheorem{cor}[thm]{Corollary}

\theoremstyle{definition}
\newtheorem{rem}[thm]{Remark}

\newcommand{\re}{\mathbb R}
\newcommand{\zz}{\mathbb Z}

\newcommand{\dd}{\,d}
\newcommand{\ex}{\mathbb E}
\newcommand{\pr}{\mathbb P}

\newcommand{\TV}{\operatorname{TV}}

\newcommand{\SL}{\operatorname{SL}}
\newcommand{\GH}{\operatorname{GH}}

\title{Optimal Sample Exponents for Direct Discrete-Gaussian SVP Search on Haar Random Lattices}

\author{Masahiro Kaminaga\\
Tohoku Gakuin University, Sendai, Japan\\
\texttt{kaminaga@g.tohoku-gakuin.ac.jp}}
\date{}

\begin{document}

\maketitle


\begin{abstract}
We determine the optimal sample exponent for direct discrete Gaussian SVP search on Haar random unimodular lattices, counting zero outputs.
The output is one sampled vector, optionally divided by the greatest common divisor of its lattice coordinates.
For every fixed approximation factor $1\leq\gamma<\sqrt e$, the exponent is $\gamma^2/(2e)-\log\gamma$ in natural logarithmic units; it is zero for $\gamma\geq\sqrt e$.
For exact SVP this gives $0.2653689\ldots$ in base two.
The converse permits arbitrary positive widths chosen from the lattice and all previous outputs, and a fixed width attains every exponent above the threshold.
Aggarwal, Dadush, Regev, and Stephens-Davidowitz already give a width within a factor two of optimal for exact SVP on each lattice.
We determine the explicit Haar typical rate and show that primitive reduction preserves it.
An explicit finite dimensional converse controls all widths simultaneously, including recovery from long multiples; together with finite attainment bounds, it yields query guarantees in both directions.
The proof uses random lattice moments and a pointwise Gaussian bound optimized over the width.
We account for sampling error and separate sample requirements from generation costs in comparisons with the random lattice search of Pouly and Shen and later algorithms.
\end{abstract}

\section{Introduction}
\label{sec:introduction}

For a full rank lattice $L\subset\re^n$, let $\lambda_1(L)$ be the length of its shortest nonzero vector.
The $\gamma$ approximate shortest vector problem asks for a nonzero vector of length at most $\gamma\lambda_1(L)$, with exact SVP corresponding to $\gamma=1$.
SVP and discrete Gaussian sampling are central to lattice algorithms and cryptography
\cite{Ajtai96,MicciancioGoldwasser02,MicciancioRegev07,Regev09,Peikert16,AjtaiKumarSivakumar01,ADRS15,SD16}.

Pouly and Shen \cite{PoulyShen26} give a rigorous Gaussian search algorithm on Haar random lattices: generate samples at a width justified by a smoothing parameter bound, and select a short nonzero sample.
They obtain time $2^{0.63269n+o(n)}$ for exact SVP and $2^{n/2+o(n)}$ for $1.123$ approximate SVP.
Their result separates two questions: how many Gaussian outputs are needed, and what does generating them cost?

Width optimization itself has a classical basis.
Following Proposition 4.3, Aggarwal, Dadush, Regev, and Stephens-Davidowitz (ADRS) \cite{ADRS15} observe that the width $s_L=1/\eta_1(L^*)$, characterized by $\theta_L(s_L)=2$, gives exact-SVP success probability within a factor two of the optimum over all widths.
Here $L^*$ is the dual lattice and $\eta_1$ is its smoothing parameter with error parameter $1$.
The same elementary comparison applies to a ball of any prescribed radius; we state and prove it in Proposition \ref{prop:balanced-width}.
Thus a converse for unmodified direct selection already follows from a deterministic width comparison once its probability at the balanced width has been evaluated.

The present problem is to determine the explicit Haar typical exponent for each approximation factor and to decide whether reducing a sampled vector to its primitive part improves that exponent.
Primitive reduction changes the successful set: arbitrarily long multiples of short primitive vectors can now succeed.
The ball comparison does not directly control this set.
We give a finite dimensional upper bound on its probability, uniform over every positive width, and match its exponential rate by a fixed width.
The resulting adaptive-query lower bound is a consequence of this common per-query bound.

Kaminaga \cite{KaminagaThermal} established primitive Gibbs visibility limits $0,1/2,1$ for fixed approximation factors $\gamma>1$.
We derive the rare-success rate, count ordinary Gaussian outputs including zero, and include exact SVP.
The main additions are the explicit approximation-dependent exponent, its preservation under primitive reduction, and finite query bounds in both directions.
The classical moment and concentration arguments needed for these statements are included below.
Their Poisson--Dirichlet edge analogue from \cite{KaminagaThermal} is not used.

Our reference ensemble is $\mathbb X_n=\SL(n,\zz)\backslash\SL(n,\re)$ with invariant probability measure $\mu_n$.
Let $v_n$ be the Euclidean unit ball volume and put $\GH_n=v_n^{-1/n}$.
For $c>0$, define
\begin{equation}
\beta_n(c)=\frac{cn}{2\GH_n^2},\qquad
s_n(c)^2=\frac{2\pi\GH_n^2}{cn}.
\label{eq:dgs-scale}
\end{equation}
The ordinary discrete Gaussian includes the zero vector:
$$
D_{L,s}(v)=\frac{e^{-\pi\|v\|^2/s^2}}{\theta_L(s)},
\qquad
\theta_L(s)=\sum_{w\in L}e^{-\pi\|w\|^2/s^2}.
$$
Write $\Theta_n(c,L)=\theta_L(s_n(c))$ and
$$
u_{n,\gamma,c}(L)=
\pr_{V\sim D_{L,s_n(c)}}\{0<\|V\|\leq\gamma\lambda_1(L)\}.
$$
For $v\neq0$, write $v=mw$ with $m\geq1$ and $w$ primitive, and define $P_L(v)=w$.
This definition is independent of the lattice basis; if $v=zB$, then $m=\gcd(z_1,\ldots,z_n)$.
Set $P_L(0)=0$ and let $\widehat u_{n,\gamma,c}(L)$ denote the success probability after applying $P_L$ to the sample.
All logarithms are natural.

\begin{thm}[Optimal sample exponent for direct Gaussian search]
\label{thm:adaptive-oracle}
Fix $\gamma\geq1$, let $L_n\sim\mu_n$, and define
\begin{equation}
K_\gamma^*=
\begin{cases}
\gamma^2/(2e)-\log\gamma,&1\leq\gamma<\sqrt e,\\
0,&\gamma\geq\sqrt e.
\end{cases}
\label{eq:optimal-raw-rate}
\end{equation}
Here $K_\gamma^*$ is measured in natural-logarithmic units.
In particular, for exact SVP,
\begin{equation}
K_1^*=\frac{1}{2e},
\qquad
\frac{K_1^*}{\log 2}
=
\frac{1}{2e\log 2}
=
0.2653689\ldots,
\label{eq:exact-base-two-rate}
\end{equation}
where the second quantity is the corresponding base-two sample exponent.
Then
\begin{equation}
-\frac1n\log\left(\sup_{c>0}u_{n,\gamma,c}(L_n)\right)
\longrightarrow K_\gamma^*
\quad\hbox{in probability}.
\label{eq:global-success}
\end{equation}
Consider a procedure making at most $N_n$ Gaussian queries, where $N_n$ is a deterministic positive integer and $n^{-1}\log N_n\to r$.
At query $j$, it chooses any $c_j>0$ measurably from $L_n$, its internal randomness, and the previous outputs.
Conditional on this history and $L_n$, the query returns a fresh draw from $D_{L_n,s_n(c_j)}$.
The procedure returns one of its nonzero samples, optionally after applying $P_{L_n}$, or reports failure.
The limit in \eqref{eq:global-success} also holds with $\widehat u$ in place of $u$.
If $r<K_\gamma^*$, its joint success probability tends to zero.
If $r>K_\gamma^*$, taking $c_j=e^{-1}$ and returning the shortest nonzero sample gives joint success tending to one.
\end{thm}

Thus ``optimal'' refers to the sample exponent within the direct search model.
Theorem \ref{thm:adaptive-oracle} therefore identifies a sharp barrier for this mechanism: neither adaptive choice of Gaussian widths nor primitive reduction improves the exponent $K_\gamma^*$.
The achieving procedure uses only sample lengths and does not require $\lambda_1(L)$.
For exact SVP the threshold in base two is $1/(2e\log2)=0.2653689\ldots$.
Factors below the exponential scale and behavior at equality are left open.
The proof is in Section \ref{sec:full-target}.
Theorem \ref{thm:finite-converse} gives an explicit finite dimensional upper bound for $\sup_{c>0}\widehat u_{n,\gamma,c}(L)$ outside a stated Haar exceptional set.
Section \ref{sec:calibration} combines this converse with an attainment bound and gives finite query budgets.

\paragraph{Algorithmic meaning and prior work.}
The output restriction defines the resource measured here: a short vector must come from one sampled ray.
The theorem evaluates this mechanism rather than the complexity of unrestricted SVP.
ADRS supply the deterministic near-optimal width comparison, and Pouly--Shen implement direct selection at a smoothing based width.
Our rate formula evaluates the Haar typical sample requirement, while the finite converse includes recovery from multiples.
Table \ref{tab:contributions} separates these results from Gaussian generation algorithms.

\begin{table}[t]
\centering
\caption{Prior results and the question settled here.
Recent preprints are identified as such; their algorithms are not assumptions of our probability theorems.}
\label{tab:contributions}
\small
\begin{tabular}{@{}p{0.22\linewidth}p{0.35\linewidth}p{0.365\linewidth}@{}}
\hline
Work & Established result or reported algorithm & Relation to this paper\\
\hline
ADRS \cite{ADRS15}
& DGS algorithms and an exact-SVP width within a factor two of optimal on every lattice.
& We evaluate the Haar typical exponent for each approximation factor and include primitive reduction.\\[3pt]
Pouly--Shen \cite{PoulyShen26}
& Random lattice SVP search with a smoothing based sampler and a time tradeoff.
& We evaluate the sample requirement at all widths and give finite bounds for the larger output model.\\[3pt]
Kaminaga \cite{KaminagaThermal} (preprint)
& Primitive Gibbs visibility limits $0,1/2,1$ for fixed $\gamma>1$.
& We add rare event rates, zero-inclusive query counts, exact SVP, and finite converse bounds.\\[3pt]
Kim \cite{Kim26DGS} (preprint)
& One arbitrary-width DGS output in expected time $2^{n/2+o(n)}$.
& This is a generation cost, distinct from the number of outputs required for direct selection.\\[3pt]
Hhan \cite{Hhan26Coset} (preprint)
& Exact SVP in time $2^{n/2+o(n)}$ using a coset difference tree.
& Combining samples and recovering a vector from Gaussian statistics falls outside direct selection.\\
\hline
\end{tabular}
\end{table}

Section \ref{sec:oracle-cost} recovers the Pouly--Shen tradeoff and compares generation costs with both that benchmark and the later algorithmic results.
The Haar ensemble is the reference model used here and in \cite{PoulyShen26}; conclusions for structured cryptographic lattices require corresponding distributional estimates.

\paragraph{Why the optimal temperature differs from the visibility threshold.}
A colder Gaussian favors short vectors conditional on a nonzero output, but also gives greater mass to zero.
For $0<c<\gamma^{-2}$, we prove the conditional success exponent
$$
\Delta_\gamma(c)=\frac12\{c\gamma^2-1-\log(c\gamma^2)\}.
$$
Adding the nonzero generation exponent gives $K_\gamma(c)$ in Theorem \ref{thm:raw-success}.
For $1\leq\gamma<\sqrt e$, its unique minimum over $0<c<1$ is at $c=e^{-1}$, below the conditional visibility threshold.
The uniform estimate in Theorem \ref{thm:adaptive-oracle} proves that widths outside this regime, or adaptive changes of width, cannot improve the exponent.

\paragraph{Proof structure.}
The fixed-width rate follows from radial volume and Gaussian weight, with moment bounds controlling the random numerator and denominator.
The converse uses Poisson summation to bound the denominator below by $\max\{1,s^n\}$.
Optimizing each vector's resulting upper bound over $s$ gives an integrable bound on a truncated annulus.
Its primitive first moment separates direct hits from recovery through multiples $m\geq2$.
The latter contribution has rate $2^{-n+o(n)}$ in the relevant regime, below the successful direct mass.
This gives explicit finite constants and a common bound for all widths without a temperature mesh in the converse.
Laarhoven \cite{Laarhoven26Spherical} studies pair and collision statistics of complete thin shells, involving relations between distinct vectors.

Sections \ref{sec:random-input}--\ref{sec:full-target} prove the main theorem.
Section \ref{sec:oracle-cost} treats generation costs and sampling accuracy.
Section \ref{sec:calibration} gives finite query bounds and numerical certificates.

\section{Random lattice thermal estimates}
\label{sec:random-input}

We specify the thermal notation and recall the moment estimates used below.
Let $L^\sharp=(L\setminus\{0\})/\{\pm1\}$, and let $L^\sharp_{\rm prim}$ be its primitive sign classes.
A nonzero vector is primitive if it is not an integer multiple $mw$ of another lattice vector with $m\geq2$.
For a sign class $u$, its norm is the norm of either representative, and
$$
\tau_n(u)=v_n\|u\|^n,\qquad
Y_n(u)=\frac1n\log\tau_n(u)=\log\frac{\|u\|}{\GH_n},
$$
$$
g_{n,c}(t)=\exp\left\{-\frac{cn}{2}(t^{2/n}-1)\right\}.
$$
In particular, $g_{n,c}(\tau_n(u))=e^{cn/2}e^{-\beta_n(c)\|u\|^2}$.
Let $\pi^{\rm prim}_{n,c,L}$ be the probability law on $L^\sharp_{\rm prim}$ proportional to this weight, and put
$$
H_{n,\rho}(L)=\{u\in L^\sharp_{\rm prim}:\|u\|\leq\rho\GH_n\},
$$
$$
G_{n,\gamma}(L)=\{u\in L^\sharp_{\rm prim}:\|u\|\leq\gamma\lambda_1(L)\}.
$$
The fixed-temperature concentration statements are from \cite{KaminagaThermal}; the proofs below also cover deterministic varying temperatures.

For a Borel set $A\subset \re$, define the primitive weighted mass
\begin{equation}
B_{n,c}(A,L)=
\sum_{\bm u\in L^\sharp_{\rm prim},\,Y_n(\bm u)\in A}
        g_{n,c}(\tau_n(\bm u)).
\label{eq:B-A}
\end{equation}
We write $B_{n,c}(L)=B_{n,c}(\re,L)$.
The corresponding primitive Gibbs mass of $A$ is
\begin{equation}
\Pi_{n,c}(A,L)=\frac{B_{n,c}(A,L)}{B_{n,c}(L)}.
\label{eq:Pi-A}
\end{equation}
Thus, for example,
$$
\Pi_{n,c}((-\infty,a],L)
$$
is the primitive Gibbs mass of the vectors satisfying $Y_n(\bm u)\leq a$.

The entropy--energy exponent behind the first moment is
\begin{equation}
\Phi_c(y)=y-\frac{c}{2}(e^{2y}-1).
\label{eq:Phi}
\end{equation}
For $0<c<1$ this function has its unique maximum at
$$
y_c=\frac{1}{2}\log\frac{1}{c},
$$
and the maximum value is
\begin{equation}
F(c)=\Phi_c(y_c)
=\frac{1}{2}\log\frac{1}{c}-\frac{1-c}{2}.
\label{eq:F}
\end{equation}

\begin{lem}[Primitive moments]
\label{lem:primitive-moments}
Let $n\geq3$ and let $f:\re^n\to[0,\infty)$ be Borel and even, with $f,f^2$ integrable, and put $S_f(L)=\sum_{u\in L_{\rm prim}^\sharp}f(u)$.
Then
\begin{equation}
\label{eq:primitive-moments}
\ex_{\mu_n}S_f=
\frac{1}{2\zeta(n)}\int_{\re^n}f(x)\,dx,
\qquad
\operatorname{Var}_{\mu_n}S_f=
\frac{1}{2\zeta(n)}\int_{\re^n}f(x)^2\,dx.
\end{equation}
\end{lem}

\begin{proof}
First take bounded, compactly supported $f$.
Rogers's formula \cite{Rogers55} gives the product integral for the sum over linearly independent nonzero pairs.
Applying M\"obius inversion separately to the two vectors changes its coefficient to
$\sum_{d,e\geq1}\mu(d)\mu(e)d^{-n}e^{-n}=\zeta(n)^{-2}$.
The interchange is justified by the corresponding absolutely summable independent-pair integrals.
Here each vector is primitive; the pair need not extend to a lattice basis.
The remaining primitive pairs are $(v,v)$ and $(v,-v)$, contributing
$\zeta(n)^{-1}\int(f(x)^2+f(x)f(-x))\,dx$ to the second moment of the sum over both signs.
The primitive first moment has coefficient $\zeta(n)^{-1}$.
Dividing the sum by two and using evenness proves \eqref{eq:primitive-moments}; see also \cite[Proposition 10]{KaminagaThermal}.
Increasing bounded compact truncations and monotone convergence extend the first and second moment identities to the stated class.
\end{proof}

\begin{lem}[Shortest length bounds]
\label{lem:shortest-bounds}
\label{lem:lambda-gh}
For $0<r<1<R$, one has
$$
\pr_{\mu_n}\{\lambda_1(L)<r\GH_n\}
\leq\frac{r^n}{2\zeta(n)},
\qquad
\pr_{\mu_n}\{\lambda_1(L)>R\GH_n\}
\leq\frac{2\zeta(n)}{R^n}.
$$
In particular, $\lambda_1(L_n)/\GH_n\to1$ in probability, and
\begin{equation}
\sqrt n\log\{\lambda_1(L_n)/\GH_n\}\to0
\quad\hbox{in probability}.
\label{eq:shortest-local}
\end{equation}
\end{lem}

\begin{proof}
The number of primitive sign classes in a ball of radius $a\GH_n$ has mean and variance $a^n/(2\zeta(n))$, by Lemma~\ref{lem:primitive-moments}.
A shortest vector is primitive.
Markov's inequality gives the first bound and Chebyshev's inequality gives the second.
For any fixed $\eta>0$, take $r=e^{-\eta/\sqrt n}$ and $R=e^{\eta/\sqrt n}$ to obtain \eqref{eq:shortest-local}.
The limiting length scale also follows from S\"odergren's Poisson theorem \cite{Sodergren11}.
\end{proof}

The fixed temperature concentration and local profile were proved in \cite[Theorem 13 and Proposition 15]{KaminagaThermal}.
We give the uniform estimate below to include deterministic temperatures depending on $n$, as arise from an exactly prescribed Gaussian width.

For the radial moment estimates, put
$$
\Psi_c(y)=y-c(e^{2y}-1),\qquad
J_\pm=n\int_{\re_\pm}e^{n\Phi_c(y)}\dd y,
\qquad K_+=n\int_0^\infty e^{n\Psi_c(y)}\dd y.
$$
Here $\re_+=[0,\infty)$ and $\re_-=(-\infty,0)$; the dependence on $n,c$ in $J_\pm,K_+$ is omitted from the notation.
By Lemma \ref{lem:primitive-moments}, the mean positive scale mass is $J_+/(2\zeta(n))$ and its variance is $K_+/(2\zeta(n))$.

For $k>0$, let
$$
P(k,x)=\frac{1}{\Gamma(k)}\int_0^x t^{k-1}e^{-t}\,dt,
\qquad Q(k,x)=1-P(k,x).
$$
Put $k=n/2$ and
$$
D_{n,c}=e^{cn/2}\Gamma(k+1)\left(\frac{2}{cn}\right)^k.
$$
Direct changes of variables give
\begin{equation}
\label{eq:gamma-integrals}
\begin{array}{rcl}
J_+&=&D_{n,c}Q(k,cn/2),\\
J_-&=&D_{n,c}P(k,cn/2),\\
K_+&=&e^{cn}\Gamma(k+1)(cn)^{-k}Q(k,cn).
\end{array}
\end{equation}
For $\rho>1$, the ratio of the full first moment integrals is
\begin{equation}
\label{eq:gamma-profile}
q_{n,c}(\rho)=P(n/2,cn\rho^2/2).
\end{equation}
The function $q_{n,c}(\rho)$ is a ratio of first moment integrals; it is not the expectation of the normalized Gibbs mass.
This is the same incomplete gamma integral used for ball masses in \cite{PoulyShen26}.
The estimates below control the random denominator as well, including temperatures at which its full relative variance does not tend to zero.

\begin{lem}[Uniform partition estimate]
\label{lem:uniform-partition}
Let $I=[c_0,c_1]\subset(0,1)$ be fixed.
Uniformly for $c\in I$, $J_-/J_+\to0$ and $K_+/J_+^2\to0$.
For every deterministic $c_n\in I$,
\begin{equation}
\frac{B_{n,c_n}(L_n)}{D_{n,c_n}/(2\zeta(n))}\to1
\quad\hbox{in probability}.
\label{eq:partition-equivalent}
\end{equation}
\end{lem}

\begin{proof}
Concavity gives $\Phi_c(y)\leq(1-c)y\leq(1-c_1)y$ for $y\leq0$, so $J_-\leq(1-c_1)^{-1}$.
The maximum of $\Phi_c$ stays in a compact subset of $(0,\infty)$, and $\Phi_c''(y_c)=-2$.
Uniform Laplace bounds give $J_+=\exp\{nF(c)+o(n)\}$.
Also
$$
S(c):=\sup_{y\geq0}\Psi_c(y)=
\begin{cases}
\frac12\log(1/(2c))-\frac12+c,&0<c<1/2,\\
0,&1/2\leq c<1,
\end{cases}
$$
and $K_+\leq\exp\{nS(c)+o(n)\}$ uniformly on $I$.
For completeness, the positive tails are bounded uniformly by replacing $c$ with $c_0$ in the negative exponential term.
On the remaining fixed compact interval the derivatives are bounded uniformly; a neighborhood of each $y_c$ gives the lower bound for $J_+$.
Thus all $o(n)$ terms can be chosen uniformly.
One has $F(c)>0$ and $2F(c)-S(c)>0$ on $(0,1)$, with positive minima on $I$.
The two ratio limits follow.

Write $B_+=B_{n,c}([0,\infty),L_n)$ and $B_-=B_{n,c}((-\infty,0),L_n)$.
Chebyshev's inequality gives $B_+/(J_+/(2\zeta(n)))\to1$; Markov's inequality gives $B_-/(J_+/(2\zeta(n)))\to0$.
Since $D_{n,c}=J_++J_-$, this proves \eqref{eq:partition-equivalent}.
\end{proof}

\begin{rem}
The split at $Y_n=0$ extends the estimate beyond the range of the untruncated second moment argument.
Its full relative variance is asymptotic to $2\zeta(n)(ec/2)^{n/2}/\sqrt{\pi n}$ and grows for $c>2/e$.
This extension supplies the complete fixed-width rate on $(0,1)$.
The optimizing temperature $e^{-1}$ already lies in the range of the untruncated argument; the extension is used for the full fixed-width rate.
\end{rem}

\begin{lem}[Thermal profile and visibility]
\label{thm:uniform-local-profile}
\label{cor:deterministic-visibility}
Let $c_n$ be deterministic in a fixed compact subinterval of $(0,1)$, and let
$y_{c_n}=\frac12\log(1/c_n)$ and $H(t)=\pi^{-1/2}\int_{-\infty}^t e^{-x^2}\dd x$.
For fixed $t\in\re$,
$$
\Pi_{n,c_n}((-\infty,y_{c_n}+t/\sqrt n],L_n)\to H(t)
\quad\hbox{in probability}.
$$
Consequently, for fixed $0<c<1$ and $\rho>1$, the mass
$\pi^{\rm prim}_{n,c,L_n}(H_{n,\rho}(L_n))$ tends to $0$, $1/2$, or $1$ according as $c<\rho^{-2}$, $c=\rho^{-2}$, or $c>\rho^{-2}$.
\end{lem}

\begin{proof}
These are the profile and visibility statements of \cite[Theorem 13 and Proposition 15]{KaminagaThermal}.
For the varying-temperature version, put $a_n=y_{c_n}+t/\sqrt n>0$ for large $n$.
The variance of $B_{n,c_n}([0,a_n],L_n)$ is at most $K_+/(2\zeta(n))$.
Lemma \ref{lem:uniform-partition}, including its bound on the negative scale mass, therefore shows that its Gibbs ratio differs by $o_{\pr}(1)$ from
$P(n/2,(n/2)e^{2t/\sqrt n})$.
The central limit theorem for a unit-scale gamma variable of shape $n/2$ gives the limit $H(t)$.
At $\log\rho=y_c$ take $t=0$; away from $y_c$, compare with $y_c+t/\sqrt n$ and let $t\to-\infty$ or $t\to\infty$.
\end{proof}

\section{A sharp success exponent for the primitive target}
\label{sec:classical-target}

For a fixed approximation factor $\gamma>1$, define the conditional one sample success mass
\begin{equation}
p_{n,\gamma,c}(L)=
\pi^{\rm prim}_{n,c,L}(G_{n,\gamma}(L)).
\label{eq:p-gamma-c}
\end{equation}
The next theorem recalls the established visibility limits and adds the matching subcritical exponent and the resulting sample count threshold.

\begin{thm}[Visibility and the subcritical success exponent]
\label{thm:target-phase}
Let $\gamma>1$ and let $L_n\sim\mu_n$.
\begin{enumerate}
\item If $\gamma^{-2}<c<1$, then
$$
p_{n,\gamma,c}(L_n)\to1
$$
in probability.
Consequently, one ideal primitive Gibbs draw is a $\gamma$ approximate shortest vector with joint probability tending to one.
At $c=\gamma^{-2}$, one has $p_{n,\gamma,c}(L_n)\to1/2$ in probability.
\item If $0<c<\gamma^{-2}$, put
\begin{equation}
\Delta_\gamma(c)=
\frac12\{c\gamma^2-1-\log(c\gamma^2)\}>0.
\label{eq:Delta-gamma}
\end{equation}
Then
$$
-\frac1n\log p_{n,\gamma,c}(L_n)
\to\Delta_\gamma(c)
$$
in probability.
\item In the setting of part 2, let $N_n$ be deterministic with $n^{-1}\log N_n\to r$.
For $N_n$ conditionally independent ideal draws, the joint success probability tends to zero if $r<\Delta_\gamma(c)$ and to one if $r>\Delta_\gamma(c)$.
No assertion is made at equality.
\end{enumerate}
\end{thm}

\begin{proof}
For part 1, $c^{-1/2}<\gamma$.
Choose $c^{-1/2}<\rho<\gamma$.
Then $c>\rho^{-2}$, and Lemma \ref{cor:deterministic-visibility} gives $\pi^{\rm prim}_{n,c,L_n}(H_{n,\rho}(L_n))\to1$ in probability.
Lemma \ref{lem:lambda-gh} gives $H_{n,\rho}(L_n)\subset G_{n,\gamma}(L_n)$ with probability tending to one.
This proves the first limit and is the form with a deterministic radius margin of the approximate SVP consequence of \cite{KaminagaThermal}.
At equality, the endpoint in the $Y_n$ coordinate is $y_c+\log\{\lambda_1(L_n)/\GH_n\}$.
Equation \eqref{eq:shortest-local} places it between $y_c-\eta/\sqrt n$ and $y_c+\eta/\sqrt n$ with probability tending to one for every fixed $\eta>0$.
Lemma \ref{thm:uniform-local-profile} and $\eta\downarrow0$ give $1/2$.
This critical limit is also established in \cite[Corollary 16]{KaminagaThermal}.

For part 2, first fix $0<a<y_c$ and $0<\ell<a$.
The number $C_n$ of primitive sign classes with $a-\ell<Y_n\leq a$ has mean and variance
$$
m_n=\frac{e^{na}-e^{n(a-\ell)}}{2\zeta(n)}.
$$
Hence $C_n/m_n\to1$ in probability by Lemma \ref{lem:primitive-moments}.
Each of these classes has weight at least $\exp\{-cn(e^{2a}-1)/2\}$, so
$$
\frac1n\log B_{n,c}(( -\infty,a],L_n)
\geq\Phi_c(a)-o_{\pr}(1).
$$
The reverse exponential bound follows from Markov's inequality and $n\int_{-\infty}^a e^{n\Phi_c(y)}\dd y =\exp\{n\Phi_c(a)+o(n)\}$.
The denominator has logarithmic rate $F(c)$ by Lemma \ref{lem:uniform-partition}.

Choose fixed radii $1<\rho_-<\gamma<\rho_+<c^{-1/2}$.
Lemma \ref{lem:lambda-gh} gives
$$
H_{n,\rho_-}(L_n)\subset G_{n,\gamma}(L_n)
\subset H_{n,\rho_+}(L_n)
$$
with probability tending to one.
Applying the preceding rate at $a=\log\rho_\pm$ and then letting $\rho_\pm\to\gamma$ yields
$$
-\frac1n\log p_{n,\gamma,c}(L_n)
\to F(c)-\Phi_c(\log\gamma)=\Delta_\gamma(c).
$$
Finally, conditional success after $N_n$ independent draws is $1-(1-p_{n,\gamma,c}(L_n))^{N_n}$.
The bounds
$$
1-(1-p)^N\leq Np,
\qquad
(1-p)^N\leq e^{-Np}
$$
give the two limits in part 3, first in probability over lattices and then after averaging.
\end{proof}

\section{Nonzero targets and ordinary Gaussian draws}
\label{sec:full-target}

For a fixed lattice and $\beta>0$, let $\pi^\sharp_{L,\beta}$ and $\pi^{\rm prim}_{L,\beta}$ be the laws proportional to $e^{-\beta\|u\|^2}$ on $L^\sharp$ and $L^\sharp_{\rm prim}$, respectively.
Extend the latter by zero outside the primitive classes.
At $\beta=\beta_n(c)$ it equals $\pi^{\rm prim}_{n,c,L}$.

\begin{prop}[Balanced width comparison]
\label{prop:balanced-width}
Let $L$ be a full rank lattice and let $s_L>0$ be the unique width with $\theta_L(s_L)=2$.
For $b>0$, put $q_b(s)=\pr_{D_{L,s}}\{0<\|V\|\leq b\}$.
Then
\begin{equation}
q_b(s_L)\leq\sup_{s>0}q_b(s)\leq2q_b(s_L).
\label{eq:balanced-width}
\end{equation}
\end{prop}

\begin{proof}
The function $\theta_L$ is continuous and strictly increasing from $1$ to infinity, so $s_L$ exists and is unique.
Write $A_b(s)=\sum_{0<\|v\|\leq b}e^{-\pi\|v\|^2/s^2}$ and $B(s)=\theta_L(s)-1$.
For $s\leq s_L$, $q_b(s)\leq A_b(s_L)=2q_b(s_L)$.
For $s\geq s_L$, the conditional ball probability $A_b(s)/B(s)$ is nonincreasing in $s$.
Indeed, when $s$ increases the ratio of the two nonzero Gaussian weights is an increasing function of $\|v\|$; summing over pairs inside and outside the ball gives this monotonicity.
Thus $q_b(s)\leq A_b(s)/B(s)\leq A_b(s_L)/B(s_L)=2q_b(s_L)$.
The first inequality in \eqref{eq:balanced-width} is immediate.
\end{proof}

For $b=\lambda_1(L)$ this is the observation following \cite[Proposition 4.3]{ADRS15}, with $s_L=1/\eta_1(L^*)$.
We include the elementary ball version to distinguish this known width comparison from the explicit Haar exponent and the treatment of primitive reduction.
For the latter operation the successful set also contains long multiples and is not a ball.

\begin{prop}[A deterministic primitive comparison]
\label{prop:full-comparison}
Put $a=\beta\lambda_1(L)^2$ and
$$
r(a)=\sum_{m=2}^\infty e^{-a(m^2-1)}.
$$
Then
\begin{equation}
\label{eq:full-tv}
\|\pi^\sharp_{L,\beta}-\pi^{\rm prim}_{L,\beta}\|_{\rm TV}
\leq\frac{r(a)}{1+r(a)}
\leq\frac{e^{-3a}}{1-e^{-5a}}.
\end{equation}
For $b>0$, let $P^\sharp(b)$ and $P^{\rm prim}(b)$ be the respective probabilities of $\|u\|\leq b$.
Then
\begin{equation}
\frac{P^{\rm prim}(b)}{1+r(a)}
\leq P^\sharp(b)\leq(1+r(a))P^{\rm prim}(b).
\label{eq:full-relative}
\end{equation}
\end{prop}

\begin{proof}
Every nonzero sign class has a unique representation $mu$, with $m\geq1$ and $u$ primitive.
If $W$ is the total primitive weight and $V$ is the total nonprimitive weight, then
$$
V=\sum_{u\in L_{\rm prim}^\sharp}e^{-\beta\|u\|^2}
\sum_{m=2}^\infty e^{-\beta(m^2-1)\|u\|^2}
\leq Wr(a).
$$
The primitive law is the full law conditioned on primitiveness.
Their total variation distance is $V/(W+V)$, which gives the first bound.
For $m\geq2$, $m^2-1\geq3+5(m-2)$, so a geometric-series bound gives the last inequality in \eqref{eq:full-tv}.
Let $W_b$ and $V_b$ be the primitive and nonprimitive weights within radius $b$.
If a multiple lies within this radius, its primitive vector lies there as well.
Thus $V_b\leq r(a)W_b$, while $V\leq r(a)W$.
Apply these bounds to $(W_b+V_b)/(W+V)$ to obtain \eqref{eq:full-relative}.
The argument also covers $W_b=0$.
\end{proof}

For an ordinary Gaussian draw $V$ at inverse temperature $\beta$, the same argument gives, for every $b>0$,
\begin{equation}
\label{eq:primitive-reduction-relative}
\pr\{0<\|V\|\leq b\}
\leq\pr\{0<\|P_L(V)\|\leq b\}
\leq(1+r(a))\pr\{0<\|V\|\leq b\}.
\end{equation}
Indeed, the middle numerator sums the weights of all positive multiples of primitive vectors of norm at most $b$, with both signs counted.
It is at most $1+r(a)$ times the primitive weight in that ball, whereas the first numerator is at least that weight; the common denominator includes zero.
This comparison is relative, so it preserves rare event exponents whenever $a$ grows linearly with $n$.

\begin{cor}[Transfer to all nonzero vectors]
\label{cor:full-transfer}
For fixed $c>0$ and $0<\kappa<1$, except on a set of Haar probability at most $\kappa^n/(2\zeta(n))$,
$$
\|\pi^\sharp_{L_n,\beta_n(c)}-\pi^{\rm prim}_{n,c,L_n}\|_{\TV}
\leq\frac{e^{-3cn\kappa^2/2}}{1-e^{-5cn\kappa^2/2}}.
$$
Thus, for $0<c<1$, the fixed factor visibility limits $0,1/2,1$ transfer to the full nonzero target, as does the local profile for $c_n$ in a compact subset of $(0,1)$.
Moreover, the full nonzero mass of $\|u\|\leq\gamma\lambda_1(L_n)$ divided by $p_{n,\gamma,c}(L_n)$ tends to one in probability.
In particular, it has the same subcritical exponent $\Delta_\gamma(c)$.
\end{cor}

\begin{proof}
On $\lambda_1(L_n)\geq\kappa\GH_n$, use $\beta_n(c)\lambda_1(L_n)^2\geq cn\kappa^2/2$ in Proposition \ref{prop:full-comparison}.
Lemma \ref{lem:shortest-bounds} bounds the exceptional probability.
Total variation controls every length event.
Equation \eqref{eq:full-relative} gives the relative comparison, since $r(a)$ tends to zero on the same set of lattices.
\end{proof}

\subsection{Ordinary discrete Gaussians and conditioning cost}
\label{sec:dgs-conditioning}

Recall the Gaussian convention and width in \eqref{eq:dgs-scale}.
Condition a draw from $D_{L,s_n(c)}$ on being primitive and then identify its sign.
The resulting law is exactly $\pi^{\rm prim}_{n,c,L}$.
For a basis $L=\zz^nB$, primitiveness of $v=zB$ is equivalent to $\gcd(z_1,\ldots,z_n)=1$.
Dividing each nonzero sample by the gcd is a different transformation: the resulting weight at a primitive class is proportional to $\sum_{m\geq1}e^{-\beta_n(c)m^2\|u\|^2}$.

\begin{thm}[Acceptance probability of primitive conditioning]
\label{thm:dgs-conditioning}
Fix $0<c<1$, and put
$$
A_n=s_n(c)^n
=\Gamma(n/2+1)\left(\frac{2}{cn}\right)^{n/2}.
$$
Let $a_n(L)$ be the probability that an ordinary $D_{L,s_n(c)}$ sample is primitive.
Then, for $L_n\sim\mu_n$,
\begin{equation}
\label{eq:dgs-acceptance}
a_n(L_n)=\frac{A_n}{1+A_n}\{1+o_{\pr}(1)\}.
\end{equation}
In particular, $a_n(L_n)\to1$ for $0<c\leq e^{-1}$, whereas for $e^{-1}<c<1$,
\begin{equation}
\label{eq:dgs-rejection}
-\frac1n\log a_n(L_n)
\longrightarrow\frac{1+\log c}{2}
\quad\hbox{in probability}.
\end{equation}
\end{thm}

\begin{proof}
Let $W_n(L)$ be the Gaussian weight of all primitive vectors with both signs.
Lemma~\ref{lem:uniform-partition} gives
$$
W_n(L_n)=2e^{-cn/2}B_{n,c}(L_n)
=\frac{A_n}{\zeta(n)}\{1+o_{\pr}(1)\}
=A_n\{1+o_{\pr}(1)\}.
$$
By Proposition~\ref{prop:full-comparison} and Lemma~\ref{lem:shortest-bounds}, the weight $V_n$ of nonprimitive nonzero vectors satisfies $V_n/W_n\to0$ in probability.
The zero vector has weight one, so
$$
a_n(L_n)=\frac{W_n}{1+W_n+V_n}.
$$
The relative approximation \eqref{eq:dgs-acceptance} follows for both small and large $A_n$.
Stirling's formula gives
$$
A_n=\sqrt{\pi n}(ec)^{-n/2}\{1+O(n^{-1})\}.
$$
At $c=e^{-1}$ this tends to infinity, as it also does for $c<e^{-1}$.
For $c>e^{-1}$ it decays exponentially, giving \eqref{eq:dgs-rejection}.
\end{proof}

Stirling's formula and \eqref{eq:dgs-scale} give the useful conversion
\begin{equation}
s_n(c)\longrightarrow(ec)^{-1/2}.
\label{eq:width-limit}
\end{equation}
The optimum $c=e^{-1}$ therefore gives a width tending to $1$, but its $n$th power tends to infinity.
Replacing this width by exactly $1$ preserves the exponential rate but changes the acceptance limit to $1/2$.
Indeed, use the deterministic temperature $c_n=2\pi\GH_n^2/n\to e^{-1}$ in Lemma \ref{lem:uniform-partition} and the preceding proof, with $A_n=1$.
Thus asymptotically equal widths need not have the same finite probability behavior.

Zero is responsible for the exponential loss when $c>e^{-1}$.
The nonprimitive nonzero mass is relatively negligible by Proposition \ref{prop:full-comparison}.
Conditioning an approximate Gaussian law on a rare nonzero event can amplify its total variation error, so Section \ref{sec:query-accuracy} states accuracy conditions for the original query law.

\subsection{Optimizing the total number of Gaussian draws}
\label{sec:raw-draws}

High success after conditioning is one possible objective.
To count all ordinary Gaussian draws, it is also necessary to allow temperatures with a small conditional success mass.
For $\gamma\geq1$, use the formula for $\Delta_\gamma(c)$ in \eqref{eq:Delta-gamma} and define
$$
D_\gamma(c)=
\begin{cases}
\Delta_\gamma(c),&0<c<\gamma^{-2},\\
0,&\gamma^{-2}\leq c<1,
\end{cases}
\qquad
I(c)=\max\left\{0,\frac{1+\log c}{2}\right\}.
$$

\begin{thm}[Success exponent for ordinary Gaussian draws]
\label{thm:raw-success}
Fix $\gamma\geq1$ and $0<c<1$.
For $L_n\sim\mu_n$,
\begin{equation}
-\frac1n\log u_{n,\gamma,c}(L_n)
\to K_\gamma(c):=I(c)+D_\gamma(c)
\quad\hbox{in probability}.
\label{eq:raw-rate}
\end{equation}
If $N_n$ is deterministic, $n^{-1}\log N_n\to r$, and the draws are independent conditional on $L_n$, the joint probability of at least one successful draw tends to zero for $r<K_\gamma(c)$ and to one for $r>K_\gamma(c)$.
The minimum of $K_\gamma(c)$ over $0<c<1$ equals $K_\gamma^*$ in \eqref{eq:optimal-raw-rate}.
For $1\leq\gamma<\sqrt e$, the unique minimizer is $c=e^{-1}$.
For $\gamma>\sqrt e$, every $\gamma^{-2}<c\leq e^{-1}$ gives $u_{n,\gamma,c}(L_n)\to1$.
At $\gamma=\sqrt e$ and $c=e^{-1}$, the limit is $1/2$.
\end{thm}

\begin{proof}
For $\gamma>1$, use the weights $W_n,V_n$ in the proof of Theorem \ref{thm:dgs-conditioning}.
The probability of a nonzero draw is
$$
\frac{W_n+V_n}{1+W_n+V_n}
=\frac{A_n}{1+A_n}\{1+o_{\pr}(1)\}.
$$
Conditional on a nonzero draw and after identifying signs, its law is the full nonzero target.
By \eqref{eq:full-relative}, its success mass equals $p_{n,\gamma,c}(L_n)\{1+o_{\pr}(1)\}$.
Multiplication and Theorem \ref{thm:target-phase} prove \eqref{eq:raw-rate} for $\gamma>1$, including its critical temperature.

For $\gamma=1$, let $M_n(L)$ be the number of shortest nonzero vectors, counting both signs.
Let $C_{n,R}(L)$ count all nonzero vectors in the ball of radius $R\GH_n$.
Decomposing each vector into a positive integer multiple of a primitive vector and applying Lemma \ref{lem:primitive-moments} gives
$$
\ex C_{n,R}=\frac{R^n}{\zeta(n)}\sum_{m=1}^\infty m^{-n}=R^n.
$$
On $\lambda_1(L_n)\leq R\GH_n$, one has $2\leq M_n\leq C_{n,R}$.
For every fixed $R>1$ and $\varepsilon>0$, Markov's inequality gives
$$
\pr\{C_{n,R}>e^{n\varepsilon}R^n\}\leq e^{-n\varepsilon}.
$$
Lemma \ref{lem:shortest-bounds}, followed by $R\downarrow1$ and $\varepsilon\downarrow0$, therefore gives $n^{-1}\log M_n(L_n)\to0$ in probability.
The Gaussian mass of shortest vectors is
$$
M_n(L_n)\exp\left\{-\frac{cn\lambda_1(L_n)^2}{2\GH_n^2}\right\},
$$
whose logarithm divided by $n$ tends to $-c/2$.
The proof of Theorem \ref{thm:dgs-conditioning} gives
\begin{equation}
\frac1n\log\Theta_n(c,L_n)\to
G(c):=\max\{0,-(1+\log c)/2\}.
\label{eq:theta-rate}
\end{equation}
Thus the exact SVP exponent is $c/2+G(c)=I(c)+\Delta_1(c)$, proving \eqref{eq:raw-rate} also for $\gamma=1$.
Here the formula \eqref{eq:Delta-gamma} defines $\Delta_1(c)$ for $0<c<1$.
The independent draw statement follows from the bounds on $1-(1-u)^{N_n}$ used in Theorem \ref{thm:target-phase}.

Suppose $1\leq\gamma<\sqrt e$.
For $c<e^{-1}$, one has $K_\gamma(c)=\Delta_\gamma(c)$ with derivative $(\gamma^2-c^{-1})/2<0$.
For $e^{-1}<c<\gamma^{-2}$, cancellation gives
$$
K_\gamma(c)=\frac{c\gamma^2}{2}-\log\gamma,
$$
which is strictly increasing.
On the remaining interval $\gamma^{-2}\leq c<1$, when it is nonempty, the rate is $(1+\log c)/2$, also strictly increasing.
The unique minimum is therefore at $e^{-1}$.
For $\gamma\geq\sqrt e$, both nonnegative terms vanish on $[\gamma^{-2},e^{-1}]$.
The success limits follow by multiplying the limits for nonzero acceptance and conditional success.
\end{proof}

\begin{figure}[t]
\centering
\begin{tikzpicture}[x=9cm,y=4.5cm]
\draw[->] (0.08,0) -- (1.03,0) node[right] {$c$};
\draw[->] (0.10,0) -- (0.10,0.66) node[above] {Exponent};
\foreach \x in {0.2,0.4,0.6,0.8,1.0}{
\draw (\x,0) -- (\x,-0.015) node[below] {\small $\x$};}
\foreach \y in {0.2,0.4,0.6}{
\draw (0.10,\y) -- (0.087,\y) node[left] {\small $\y$};}
\draw[densely dashed,gray] (0.367879,0) -- (0.367879,0.59);
\draw[densely dashed,gray] (0.694444,0) -- (0.694444,0.59);
\node[above] at (0.367879,0.59) {\small $e^{-1}$};
\node[above] at (0.694444,0.59) {\small $\gamma^{-2}$};
\draw[blue,thick,domain=0.10:0.367879,samples=80]
plot (\x,{0.5*(1.44*\x-1-ln(1.44*\x))});
\draw[blue,thick,domain=0.367879:0.694444,samples=30]
plot (\x,{0.72*\x-ln(1.2)});
\draw[blue,thick,domain=0.694444:1,samples=30]
plot (\x,{0.5*(1+ln(\x))});
\draw[black,dashed,domain=0.10:0.694444,samples=80]
plot (\x,{0.5*(1.44*\x-1-ln(1.44*\x))});
\draw[black,dashed] (0.694444,0) -- (1,0);
\draw[red,dotted,thick] (0.10,0) -- (0.367879,0);
\draw[red,dotted,thick,domain=0.367879:1,samples=60]
plot (\x,{0.5*(1+ln(\x))});
\fill[blue] (0.367879,0.082552) circle (1.8pt);
\node[anchor=west,blue] at (0.46,0.51) {\small $K_\gamma(c)$};
\node[anchor=west] at (0.12,0.24) {\small $D_\gamma(c)$};
\node[anchor=west,red] at (0.76,0.29) {\small $I(c)$};
\end{tikzpicture}
\caption{Exponents at $\gamma=1.2$.
The solid curve is the rate $K_\gamma(c)=D_\gamma(c)+I(c)$ for ordinary Gaussian draws.
The dashed curve is the rate for the accepted primitive target, and the dotted curve is the conditioning rate.
The minimum total rate is attained at $c=e^{-1}$, below the threshold for high conditional success.
The curves are evaluations of the proved formulas.}
\label{fig:draw-rates}
\end{figure}
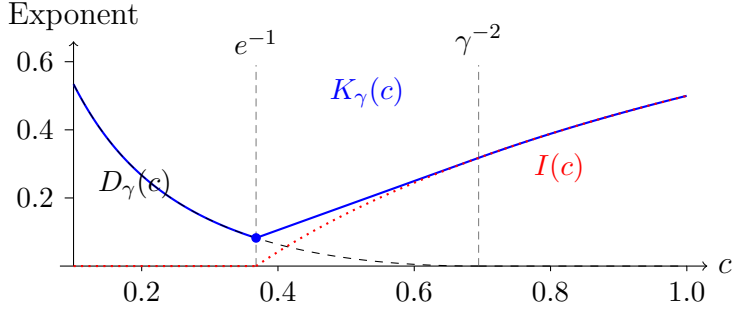

For $1<\gamma<\sqrt e$, the difference between the best exponent under high conditional success and the optimized rate is
$$
\left(\frac12-\log\gamma\right)
-\left(\frac{\gamma^2}{2e}-\log\gamma\right)
=\frac{1-\gamma^2/e}{2}>0.
$$
Figure \ref{fig:draw-rates} illustrates this comparison.
The next result makes the fixed-width rate uniform on compact temperature intervals, as needed for the varying widths in Section \ref{sec:oracle-cost}.

\begin{lem}[Uniform Gaussian success exponent]
\label{lem:uniform-raw}
Let $\gamma\geq1$ and let $I\subset(0,1)$ be a fixed compact interval.
Then
\begin{equation}
\label{eq:uniform-raw}
\sup_{c\in I}\left|\frac1n\log u_{n,\gamma,c}(L_n)+K_\gamma(c)\right|
\to0
\quad\hbox{in probability}.
\end{equation}
The same statement holds with $\widehat u$ in place of $u$.
\end{lem}

\begin{proof}
Put
$$
S_n(c,L)=\sum_{0<\|v\|\leq\gamma\lambda_1(L)}
             e^{-\beta_n(c)\|v\|^2}.
$$
The functions $S_n(c,L)$ and $\Theta_n(c,L)$ are positive and nonincreasing in $c$ for every fixed lattice.
Equation \eqref{eq:theta-rate} gives the pointwise logarithmic limit $G(c)$ for $\Theta_n$.
Since $S_n=u_{n,\gamma,c}\Theta_n$, Theorem \ref{thm:raw-success} gives
$$
\frac1n\log S_n(c,L_n)\to G(c)-K_\gamma(c).
$$
The two limits are continuous on $I$.
Choose a finite mesh on which the oscillation of each limit between adjacent points is at most $\varepsilon$.
Convergence at the finitely many mesh points holds simultaneously outside an event of probability tending to zero.
Monotonicity bounds each logarithm between its values at the two neighboring mesh points.
Its uniform error is therefore at most the largest mesh error plus $\varepsilon$.
Letting $\varepsilon\downarrow0$ proves uniform convergence of both logarithms.
Their difference gives \eqref{eq:uniform-raw}.
The same limit holds with $\widehat u$ in place of $u$.
Indeed, on $\lambda_1(L_n)\geq\kappa\GH_n$ for any fixed $0<\kappa<1$, the factor in \eqref{eq:primitive-reduction-relative} tends to one uniformly for $c\in I$, since $a\geq n\kappa^2\min I/2$.
Lemma \ref{lem:shortest-bounds} makes the exceptional probability tend to zero.
\end{proof}

\subsection{An explicit converse uniform over all widths}
\label{sec:finite-converse}

The next bound treats primitive reduction directly and has explicit finite dimensional constants.
It uses only the first primitive moment, the shortest length bounds, and Poisson summation.

\begin{thm}[Finite converse with primitive reduction]
\label{thm:finite-converse}
Let $n\geq3$, $\gamma\geq1$, $k=n/2$, and
$$
\alpha_n=\frac{\sqrt{n/(2\pi)}}{\GH_n}.
$$
Choose $0<\eta<1$, $0<\kappa<1<R$, and put $\rho=\gamma R$.
Assume $\alpha_n/2\leq\kappa$ and $\rho\leq\alpha_n$.
Define
\begin{eqnarray}
M_1&=&\frac{P(k,\pi\rho^2\GH_n^2)-P(k,\pi\kappa^2\GH_n^2)}{\zeta(n)},\nonumber\\
M_2&=&\frac{n(k/e)^k}{\Gamma(k+1)}
       \frac{\zeta(n)-1}{\zeta(n)}\log\frac{\rho}{\kappa},\nonumber\\
U&=&\min\{1,(M_1+M_2)/\eta\},\qquad
E_+=\eta+\frac{\kappa^n}{2\zeta(n)}+\frac{2\zeta(n)}{R^n}.
\label{eq:finite-converse-constants}
\end{eqnarray}
Then
\begin{equation}
\pr_{\mu_n}\left\{\sup_{c>0}\widehat u_{n,\gamma,c}(L)>U\right\}\leq E_+.
\label{eq:finite-converse}
\end{equation}
The same upper bound holds without primitive reduction.
\end{thm}

\begin{proof}
Poisson summation for a unimodular lattice gives
$\theta_L(s)=s^n\theta_{L^*}(1/s)\geq\max\{1,s^n\}$.
For $r>0$, direct maximization gives
\begin{equation}
H_n(r):=\sup_{s>0}\frac{e^{-\pi r^2/s^2}}{\max\{1,s^n\}}
=\begin{cases}
e^{-\pi r^2},&r\leq\sqrt{n/(2\pi)},\\
\bigl(n/(2\pi e r^2)\bigr)^{n/2},&r\geq\sqrt{n/(2\pi)}.
\end{cases}
\label{eq:pointwise-envelope}
\end{equation}
For $s\leq1$ the numerator is increasing in $s$.
For $s\geq1$ the logarithmic derivative is $-n/s+2\pi r^2/s^3$, giving \eqref{eq:pointwise-envelope}.

On $\kappa\GH_n\leq\lambda_1(L)\leq R\GH_n$, every successful primitive vector $w$, counting both signs, satisfies $\kappa\GH_n\leq\|w\|\leq\rho\GH_n$.
Consequently, simultaneously for every $s>0$, its successful mass is bounded above by
$$
Z(L):=2\sum_{w\in L^\sharp_{\rm prim},\,
                 \kappa\GH_n\leq\|w\|\leq\rho\GH_n}
                 \sum_{m\geq1}H_n(m\|w\|).
$$
The factor $2$ counts the two signs.
The hypotheses on $\kappa,\rho$ imply that $m=1$ uses the first branch of \eqref{eq:pointwise-envelope}, whereas every $m\geq2$ uses the second branch.
Thus, for $\kappa\GH_n\leq r\leq\rho\GH_n$,
$$
\sum_{m\geq1}H_n(mr)
=e^{-\pi r^2}+(\zeta(n)-1)\left(\frac{n}{2\pi e}\right)^{n/2}r^{-n}.
$$
Apply the primitive first moment in Lemma \ref{lem:primitive-moments}.
The radial Gaussian integral is the difference of gamma probabilities in $M_1$.
The second integral is
$$
\frac{\zeta(n)-1}{\zeta(n)}
\left(\frac{n}{2\pi e}\right)^{n/2}
n v_n\int_{\kappa\GH_n}^{\rho\GH_n}\frac{dr}{r}=M_2.
$$
All sums are nonnegative, and the annulus is bounded away from zero, so these interchanges are justified by Tonelli's theorem.
We obtain $\ex Z=M_1+M_2$.
Markov's inequality and Lemma \ref{lem:shortest-bounds} give \eqref{eq:finite-converse} by a union bound.
Clipping at $1$ preserves the upper bound.
\end{proof}

\begin{cor}[Matching exponential rate of the finite converse]
\label{cor:finite-converse-rate}
Fix $1\leq\gamma<\sqrt e$.
In Theorem \ref{thm:finite-converse}, take
$$
\eta_n=n^{-1},\qquad \kappa_n=n^{-1/n},\qquad R_n=n^{1/n}.
$$
Its hypotheses hold for all sufficiently large $n$, $E_{+,n}\to0$, and
$$
-\frac1n\log U_n\longrightarrow K_\gamma^*.
$$
\end{cor}

\begin{proof}
Stirling's formula gives $\alpha_n\to\sqrt e$, so the geometric hypotheses hold eventually.
The three terms in $E_{+,n}$ tend to zero.
Writing $r=x\GH_n$ in the first moment gives
$$
M_1=\frac{n}{\zeta(n)}\int_{\kappa_n}^{\gamma R_n}
          x^{n-1}e^{-\pi\GH_n^2x^2}\,dx.
$$
Here $\pi\GH_n^2/n\to1/(2e)$.
For $\gamma>1$, the limiting exponent $\log x-x^2/(2e)$ is strictly increasing on $[1,\gamma]$.
An upper bound by its maximum and a lower bound from a fixed interval near $\gamma$ give
$n^{-1}\log M_1\to\log\gamma-\gamma^2/(2e)$.
For $\gamma=1$, both endpoints tend to $1$, the integrand has rate $-1/(2e)$ uniformly on the interval, and the interval length is asymptotic to $2\log n/n$.
The same limit follows.
Finally, $\zeta(n)-1=2^{-n}\{1+o(1)\}$ and
$n(k/e)^k/\Gamma(k+1)=\sqrt{n/\pi}\{1+O(n^{-1})\}$.
The logarithmic factor in $M_2$ is bounded above and below by subexponential factors, so $M_2=2^{-n+o(n)}$.
Since $K_\gamma^*\leq1/(2e)<\log2$, $M_2$ does not affect the rate of $M_1+M_2$.
Division by $\eta_n$ and clipping at $1$ give the assertion.
\end{proof}

\begin{proof}[Proof of Theorem \ref{thm:adaptive-oracle}]
For $1\leq\gamma<\sqrt e$, Theorem \ref{thm:finite-converse} and Corollary \ref{cor:finite-converse-rate} give, on a common lattice set of probability tending to one,
$$
\sup_{c>0}\widehat u_{n,\gamma,c}(L)\leq e^{-n(K_\gamma^*-\delta)}
$$
for every fixed $\delta>0$ and all sufficiently large $n$.
Theorem \ref{thm:raw-success} at $c=e^{-1}$ gives the matching rate without primitive reduction.
Since $u\leq\widehat u$, both suprema have rate $K_\gamma^*$.
For $\gamma\geq\sqrt e$, the bound by $1$ and the same fixed-width rate give the two limits with exponent zero.

The common per-query bound holds conditional on every history and every allowed choice of width.
For $r<K_\gamma^*$ choose $0<\delta<K_\gamma^*-r$.
The union bound gives conditional success at most $N_ne^{-n(K_\gamma^*-\delta)}$ on the good set, and its complement has probability tending to zero.
For $r>K_\gamma^*$, independent draws at $c=e^{-1}$ attain success tending to one by Theorem \ref{thm:raw-success}.
Returning the shortest nonzero sample achieves this probability without knowing $\lambda_1(L)$.
\end{proof}

\section{Cryptographic interpretation of the main theorem}
\label{sec:oracle-cost}

Theorem \ref{thm:adaptive-oracle} identifies the sample requirement of direct Gaussian SVP search.
We now combine it with generation costs to answer two algorithmic questions: how does it recover the Pouly--Shen tradeoff, and what must a sampler at the optimal width achieve to improve that tradeoff?
These are conditional cost compositions: the sampling probabilities are evaluated in the real Haar model, and generation costs are supplied by a specified implementation.
The cited finite-input samplers retain their own input conventions, including polynomial dependence on the basis encoding length.
We suppress such factors in dimension exponents only when they are $\exp(o(n))$.
A bit-complexity theorem for a finitely represented lattice ensemble additionally requires a transfer of the Haar probability estimates; no such transfer is assumed here.
All cost exponents below use natural logarithms.

\subsection{From sample exponents to search costs}
\label{sec:cost-composition}

Suppose individual independent draws at a fixed $c$ cost
$$
T_n(c)=\exp\{n\tau(c)+o(n)\},
$$
including the work needed to inspect an output.
Using a deterministic number of draws with exponent $r>K_\gamma(c)$ gives success tending to one and time exponent $\tau(c)+r$.
The infimum exponent for this implementation at $c$ is therefore $\tau(c)+K_\gamma(c)$.
A temperature minimizing $K_\gamma$ alone minimizes this sum when $\tau$ is constant, but need not do so when generation cost depends on temperature.

For batch generation, suppose one call produces $M_n(c)$ independent draws at cost $C_n(c)$, with independent calls and
$$
\frac1n\log M_n(c)\to\kappa(c),\qquad
\frac1n\log C_n(c)\to b(c).
$$
Include the cost of examining the batch in $C_n(c)$.
Using $\exp\{n\ell+o(n)\}$ calls produces a list with exponent $\ell+\kappa(c)$, so the infimum time exponent for repeated full batches is
\begin{equation}
\label{eq:batch-cost}
b(c)+\max\{0,K_\gamma(c)-\kappa(c)\}.
\end{equation}
Indeed, any $\ell>\max\{0,K_\gamma(c)-\kappa(c)\}$ gives success tending to one, whereas $\ell+\kappa(c)<K_\gamma(c)$ gives failure.
The infimum allows arbitrarily small positive slack; at equality a single batch need not suffice.
If each call has joint error at most $\eta_n$, $B_n$ independent calls change list success by at most $B_n\eta_n$.

\subsection{Comparison with the Pouly--Shen algorithm}
\label{sec:pouly-shen}

We now compare two Gaussian generation regimes from \cite{ADRS15}.
The general sampler in Theorem 3.7 of its full version produces $2^{n/2}$ samples in time $2^{n+o(n)}$.
Taking its confidence parameter to be $n^2$ makes the joint error $\exp(-\Omega(n^2))$.
Here $\kappa=(\log2)/2$ and $b=\log2$.
Since $K_\gamma^*\leq1/(2e)<(\log2)/2$ for every fixed $\gamma\geq1$, one ideal batch at $c=e^{-1}$ succeeds on typical Haar lattices.
Its construction cost remains $2^{n+o(n)}$.

Pouly and Shen \cite{PoulyShen26} instead use the faster sampler available above $\sqrt2\eta_{1/2}(L)$, where $\eta_{1/2}$ is the smoothing parameter.
Their smoothing estimate permits the deterministic width
$$
\sigma_n=\sqrt2(6+4\sqrt2)^{1/n}
$$
outside an exceptional set of Haar measure tending to zero.
This sampler has batch size and expected time $2^{n/2+o(n)}$.
For this comparison, batch cost means expected generation time plus inspection.
Its Gaussian convention is the same as \eqref{eq:dgs-scale}.
The following consequence gives a direct comparison on our logarithmic scale.

\begin{cor}[Draw rate at the smoothing based width]
\label{cor:smoothing-comparison}
Fix $\gamma\geq1$ and set
$$
c_n=\frac{2\pi\GH_n^2}{n\sigma_n^2},
\qquad c_0=\frac1{2e},
\qquad L_\gamma=K_\gamma(c_0).
$$
For $V\sim D_{L_n,\sigma_n}$, the probability of $0<\|V\|\leq\gamma\lambda_1(L_n)$ has exponent $L_\gamma$ in probability over lattices.
For $1\leq\gamma<\sqrt{2e}$,
\begin{equation}
L_\gamma=\frac{\log2}{2}-\log\gamma+\frac{\gamma^2}{4e}.
\label{eq:smoothing-draw-rate}
\end{equation}
For ideal independent batches with $b=\kappa=(\log2)/2$, the infimum time exponent for repeated full batches is
\begin{equation}
T_\gamma=\max\left\{\frac{\log2}{2},L_\gamma\right\}.
\label{eq:smoothing-batch-rate}
\end{equation}
\end{cor}

\begin{proof}
Equation \eqref{eq:width-limit} gives $c_n\to c_0$.
The temperatures lie in a fixed compact subinterval of $(0,1)$ for all large $n$.
Lemma \ref{lem:uniform-raw} and continuity of $K_\gamma$ therefore give the stated rate.
For $1\leq\gamma<\sqrt{2e}$, one has $I(c_0)=0$ and $c_0\gamma^2<1$.
Substitution in \eqref{eq:raw-rate} gives \eqref{eq:smoothing-draw-rate}.
Equation \eqref{eq:batch-cost} gives \eqref{eq:smoothing-batch-rate}.
\end{proof}

The equality $L_\gamma=(\log2)/2$ has the root $\gamma_0=1.1229739\ldots$ in $(1,\sqrt{2e})$.
For $\gamma>\gamma_0$, one ideal batch at $\sigma_n$ suffices; below that value additional batches are needed on the exponential scale.
At $\gamma=1$, equation \eqref{eq:smoothing-batch-rate} gives
$$
\frac{T_1}{\log2}=\frac12+\frac{1}{4e\log2}=0.6326845\ldots.
$$
This recovers the approximate and exact SVP time tradeoff of \cite{PoulyShen26}.
An error $\exp(-\Omega(n^2))$ in the joint law of each batch remains negligible after exponentially many calls, as required above.

\subsection{A cost target at the optimal width}
\label{sec:sampler-target}

The sample optimal width $s_n(e^{-1})$ tends to $1$, whereas the width $\sigma_n$ justified by the Pouly--Shen smoothing bound tends to $\sqrt2$.
The following consequence states precisely what a sampler at the smaller width would need to deliver.

\begin{cor}[Generation cost needed to improve the benchmark]
\label{cor:sampler-budget}
Fix $\gamma\geq1$.
Suppose independent ordinary Gaussian draws at $c=e^{-1}$ can be generated and inspected at cost $\exp\{n\tau_*+o(n)\}$ per draw.
Then the infimum time exponent for direct search with deterministic query budgets and success tending to one is $\tau_*+K_\gamma^*$.
It is below the benchmark $T_\gamma$ in \eqref{eq:smoothing-batch-rate} exactly when
\begin{equation}
\tau_*<T_\gamma-K_\gamma^*.
\label{eq:sampler-budget}
\end{equation}
For repeated full independent batches at this width with size exponent $\kappa_*$ and cost exponent $b_*$, the corresponding condition is
$$
b_*+\max\{0,K_\gamma^*-\kappa_*\}<T_\gamma.
$$
These conclusions also hold under the accuracy conditions below.
\end{cor}

\begin{proof}
Apply Theorem \ref{thm:raw-success} at $c=e^{-1}$ and the cost composition in Section \ref{sec:cost-composition}.
The batch statement follows from \eqref{eq:batch-cost}.
Strict inequalities permit a positive slack in the sample or batch exponent, giving success tending to one.
\end{proof}

For exact SVP, condition \eqref{eq:sampler-budget} becomes
$$
\frac{\tau_*}{\log2}<\frac12-\frac{1}{4e\log2}
=0.3673154\ldots.
$$
Table \ref{tab:search-costs} gives further examples.
The last column concerns improvement over the Pouly--Shen direct-selection benchmark specifically.
It is a generation cost target, not a construction meeting that target.

\begin{table}[ht]
\centering
\caption{Sample and time exponents in base two.
$K_\gamma^*$ is the optimal sample exponent, $L_\gamma$ is the sample exponent at the Pouly--Shen width, and $T_\gamma$ includes its batch cost.
The final column is the strict per-draw cost threshold in \eqref{eq:sampler-budget}.
Values are evaluations of the proved formulas.}
\label{tab:search-costs}
\begin{tabular}{rcccc}
\hline
$\gamma$ & $K_\gamma^*/\log2$ & $L_\gamma/\log2$ & $T_\gamma/\log2$ & $(T_\gamma-K_\gamma^*)/\log2$\\
\hline
1.0 & 0.26537 & 0.63268 & 0.63268 & 0.36732\\
1.1 & 0.18359 & 0.52304 & 0.52304 & 0.33945\\
1.2 & 0.11910 & 0.42803 & 0.50000 & 0.38090\\
1.4 & 0.03470 & 0.27463 & 0.50000 & 0.46530\\
\hline
\end{tabular}
\end{table}

The distinction between these columns matters for algorithm design.
Theorem \ref{thm:adaptive-oracle} rules out reducing the optimal sample exponent by changing widths alone.
A faster implementation must also control generation cost; combining vectors, as in sieving, changes the search model and can exploit information beyond the event that one sample is already short.

\subsection{Later sampling and SVP algorithms}
\label{sec:later-algorithms}

Kim \cite[Theorem 4.3]{Kim26DGS} gives an arbitrary-width one-sample DGS algorithm with expected time $2^{n/2+o(n)}$ and statistical error $\exp(-\Omega(n^3))$, for rational bases and positive rational squared widths.
Under the cost composition above, independent calls at a width with $c_n\to e^{-1}$ have achievable expected time exponents arbitrarily close to
$$
\frac12+\frac{K_\gamma^*}{\log2}
$$
in base two.
For exact SVP this is $0.7653689\ldots$, so this particular independent-call use does not improve the Pouly--Shen benchmark.
This computation concerns the black-box use of the one-sample interface, not the other SVP constructions in \cite{Kim26DGS} or all possible uses of its internal lists.
Rational squared widths with vanishing relative error preserve the rate by Lemma \ref{lem:uniform-raw}.

Hhan \cite[Theorem 4.1]{Hhan26Coset} reports a $2^{n/2+o(n)}$ time exact-SVP algorithm using coset difference trees and Gaussian statistics.
This preprint gives a separate comparison point: combining samples and recovering a vector from their statistics does not require a successful ordinary draw on the original lattice.
To beat that reported time exponent by independent direct draws at the sample optimal width, the cost model of Corollary \ref{cor:sampler-budget} instead requires
\begin{equation}
\frac{\tau_*}{\log2}<\frac12-\frac{K_\gamma^*}{\log2}.
\label{eq:later-cost-target}
\end{equation}
For exact SVP the right-hand side is $0.2346311\ldots$.
Equations \eqref{eq:sampler-budget} and \eqref{eq:later-cost-target} answer different benchmark questions.
Neither later algorithm changes our lower bound, which allows primitive reduction of each draw but does not allow combining distinct sampled vectors or replacing the queried lattice.

\begin{table}[ht]
\centering
\caption{Exact-SVP comparison points, in base-two dimension exponents.
Costs retain the input and expectation conventions discussed in this section.}
\label{tab:benchmark-roles}
\small
\begin{tabular}{@{}p{0.34\linewidth}p{0.16\linewidth}p{0.43\linewidth}@{}}
\hline
Mechanism & Time exponent & Role of the sample bound\\
\hline
Pouly--Shen direct selection & $0.63268\ldots$ & Recovered at their smoothing based width.\\[3pt]
Independent calls to Kim's one-sample interface & $0.76537\ldots$ & Conditional cost composition at the sample optimal width.\\[3pt]
Hhan's coset difference tree & $0.5$ & Separate reported algorithm; uses information beyond direct selection.\\[3pt]
Direct draws at the optimal width & $t+0.26537\ldots$ & A per-draw cost exponent $t<0.23463\ldots$ is required to beat $0.5$.\\
\hline
\end{tabular}
\end{table}

\subsection{Accuracy required of the sampler}
\label{sec:query-accuracy}

\begin{cor}[Approximate adaptive queries]
\label{cor:approx-raw}
In Theorem \ref{thm:adaptive-oracle}, replace each exact conditional query law by a law on $L_n$ at total variation distance at most $\epsilon_n$ from $D_{L_n,s_n(c_j)}$.
Assume this bound holds for every query history on a common set of lattices of probability tending to one.
If, for some $\delta>0$,
$$
\epsilon_n=o(e^{-n(K_\gamma^*+\delta)}),
$$
then both sample count conclusions of that theorem remain valid.
For a fixed $c$, the analogous condition with $K_\gamma(c)$ preserves the exponent and the threshold in Theorem \ref{thm:raw-success}.
\end{cor}

\begin{proof}
For unmodified or primitive-reduced output, the conditional success probability differs from $u_{n,\gamma,c_j}(L)$ or $\widehat u_{n,\gamma,c_j}(L)$ by at most $\epsilon_n$.
For the lower bound, the union bound is at most $N_n(\sup_{c>0}\widehat u_{n,\gamma,c}+\epsilon_n)$.
For achievability at $c=e^{-1}$, choose $0<\xi<\min\{\delta,r-K_\gamma^*\}$.
On typical lattices, every conditional success probability is at least
$$
e^{-n(K_\gamma^*+\xi)}-\epsilon_n
\geq\tfrac12 e^{-n(K_\gamma^*+\xi)}.
$$
Iteration of conditional failure probabilities gives an upper bound
$\exp\{-N_ne^{-n(K_\gamma^*+\xi)}/2\}$, which tends to zero.
For fixed $c$, the error divided by the ideal success mass tends to zero on typical lattices, so each conditional success mass has logarithmic rate $K_\gamma(c)$.
The same failure bounds prove its sample count threshold.
\end{proof}

This assumption controls the law after conditioning on previous outputs; accurate marginal laws alone do not suffice.
For example, repeating one exact Gaussian sample gives exact marginals but no gain from repetition.
Alternatively, if the joint law of an entire list is within $\eta_n=o(1)$ of independent exact draws, the list success probability differs by at most $\eta_n$.
Thus a bound on the joint law can suffice even when a bound on each separate marginal would be too weak.

\section{Finite query guarantees}
\label{sec:calibration}

The finite converse in Theorem \ref{thm:finite-converse} bounds the success of every allowed width policy.
We pair it with a finite attainment bound at a specified width, counting ordinary Gaussian draws including zero.
Both statements apply to exact SVP and fixed approximation factors.

\begin{thm}[Finite ordinary Gaussian success bounds]
\label{thm:finite-raw}
Let $n\geq3$, $c>0$, and $0<\eta<1$.
Put
\begin{equation}
A=\Gamma(n/2+1)\left(\frac2{cn}\right)^{n/2},
\qquad D=1+A/\eta.
\label{eq:finite-denominator}
\end{equation}
The following statements concern $L\sim\mu_n$.
\begin{enumerate}
\item For exact SVP, choose $R>1$ and define
\begin{equation}
q_{\rm ex}=\min\{1,2e^{-cnR^2/2}/D\},\qquad
E_{\rm ex}=\eta+2\zeta(n)R^{-n}.
\label{eq:finite-exact}
\end{equation}
Outside a lattice set of probability at most $E_{\rm ex}$, one ordinary Gaussian draw has success probability at least $q_{\rm ex}$.
\item For $\gamma>1$, choose $1<\rho<\gamma$ and $0<h<1$, and put
\begin{eqnarray}
m_\rho&=&\frac{\rho^n-1}{2\zeta(n)},\nonumber\\
q_{\rm ap}&=&\min\{1,2(1-h)m_\rho e^{-cn\rho^2/2}/D\},\nonumber\\
E_{\rm ap}&=&\eta+\frac1{h^2m_\rho}
+\frac{(\rho/\gamma)^n}{2\zeta(n)}.
\label{eq:finite-approx}
\end{eqnarray}
Outside a lattice set of probability at most $E_{\rm ap}$, one ordinary Gaussian draw has $\gamma$ approximate success probability at least $q_{\rm ap}$.
\end{enumerate}
Primitive reduction preserves both lower bounds.
Bounds with $E\geq1$ are valid but uninformative.
\end{thm}

\begin{proof}
The first moment identity, after summing primitive multiples, gives
$\ex(\Theta_n(c,L)-1)=s_n(c)^n=A$.
Markov's inequality therefore gives $\Theta_n(c,L)\leq D$ outside a set of probability at most $\eta$.
For exact SVP, Lemma \ref{lem:shortest-bounds} gives $\lambda_1(L)\leq R\GH_n$ outside a set of probability at most $2\zeta(n)R^{-n}$.
On this event the two signs of a shortest vector contribute at least $2e^{-cnR^2/2}$.

For approximate SVP, count primitive sign classes in the annulus
$\GH_n<\|v\|\leq\rho\GH_n$.
Its mean and variance are both $m_\rho$ by Lemma \ref{lem:primitive-moments}.
Chebyshev's inequality gives at least $(1-h)m_\rho$ classes outside a set of probability at most $1/(h^2m_\rho)$.
Both signs of each contribute at least $2e^{-cn\rho^2/2}$.
On $\lambda_1(L)\geq(\rho/\gamma)\GH_n$ they are successful; Lemma \ref{lem:shortest-bounds} bounds the remaining exceptional probability.
Divide each numerator bound by $D$ and use the union bound.
Primitive reduction cannot increase a nonzero vector's norm.
\end{proof}

\begin{cor}[Query budget and sampling error]
\label{cor:finite-query-budget}
Let $(q,E)$ be either pair in Theorem \ref{thm:finite-raw}.
Suppose that, outside another lattice set of probability at most $\omega$, every query law conditional on the past is supported on $L$ and is within total variation $\epsilon<q$ of $D_{L,s_n(c)}$.
For $N$ queries, returning the shortest nonzero sample, with optional primitive reduction, succeeds conditionally on every lattice in the common good set with probability at least
\begin{equation}
1-e^{-N(q-\epsilon)}.
\label{eq:finite-query-success}
\end{equation}
The joint success probability is at least
$\max\{0,1-E-\omega\}(1-e^{-N(q-\epsilon)})$.
In particular, $N\geq\log(1/\delta)/(q-\epsilon)$ guarantees conditional success at least $1-\delta$ there.
\end{cor}

\begin{proof}
Every conditional success probability is at least $q-\epsilon$.
Iteration bounds the probability of no successful draw by $(1-q+\epsilon)^N\leq e^{-N(q-\epsilon)}$.
Average over the common lattice set, whose probability is at least $\max\{0,1-E-\omega\}$.
\end{proof}

\begin{cor}[Finite adaptive query converse]
\label{cor:finite-query-converse}
Let $(U,E_+)$ be as in Theorem \ref{thm:finite-converse}.
Consider at most $N$ queries in the direct search model of Theorem \ref{thm:adaptive-oracle}.
Outside an additional lattice set of probability at most $\omega$, suppose every conditional query law is within total variation $\epsilon$ of the chosen ordinary Gaussian law.
Then its joint success probability, with or without primitive reduction, is at most
\begin{equation}
\min\{1,E_++\omega+N(U+\epsilon)\}.
\label{eq:finite-query-converse}
\end{equation}
In particular, success at least $1-\delta$ requires
$N\geq(1-\delta-E_+-\omega)/(U+\epsilon)$ whenever the numerator is positive.
\end{cor}

\begin{proof}
On the common good lattice set, every query has conditional success probability at most $U+\epsilon$, for every history and width.
The union bound gives $N(U+\epsilon)$ there; the excluded probability is at most $E_++\omega$.
\end{proof}

\paragraph{Relation to the optimal exponent.}
For $c=e^{-1}$, $A=\sqrt{\pi n}\{1+O(n^{-1})\}$.
In the exact bound, take $\eta_n=n^{-1}$ and $R_n=n^{1/n}$.
Then $E_{\rm ex}\to0$ and $-n^{-1}\log q_{\rm ex}\to1/(2e)$.
For fixed $1<\gamma<\sqrt e$, take $\eta_n=n^{-1}$, $h=1/4$, and $\rho_n=\gamma n^{-1/n}$, which is greater than one for large $n$.
Then $E_{\rm ap}\to0$ and
$$
-\frac1n\log q_{\rm ap}\to\frac{\gamma^2}{2e}-\log\gamma=K_\gamma^*.
$$
Thus these finite bounds recover the attaining exponent in the rare success regime.
Corollary \ref{cor:finite-converse-rate} gives the same limiting exponent for the finite converse, so the two finite query bounds match on the exponential scale.

\begin{table}[t]
\centering
\caption{Finite guarantees for ordinary Gaussian draws at $c=e^{-1}$.
The formulas give $q\geq2^{-a}$ outside a lattice set of probability at most the displayed $E$, rounded upward.
With $N=2^{a+3}$ and conditional query error at most $2^{-a-7}$ on every lattice, success exceeds $0.999$ on the good set.
The last column has $U\leq2^{-b}$ for the all-width converse, with $E_+<0.00351$.
These are sample budgets, excluding generation cost.}
\label{tab:finite-bounds}
\begin{tabular}{rrrrrr}
\hline
$n$ & $\gamma$ & $a$ & $\log_2N$ & $E$ & $b$\\
\hline
100 & 1.0 & 44 & 47 & 0.00301 & 11\\
300 & 1.0 & 97 & 100 & 0.00301 & 65\\
500 & 1.0 & 151 & 154 & 0.00301 & 118\\
100 & 1.2 & 31 & 34 & 0.00289 & 0\\
300 & 1.1 & 76 & 79 & 0.00251 & 41\\
300 & 1.2 & 55 & 58 & 0.00251 & 23\\
500 & 1.2 & 79 & 82 & 0.00251 & 47\\
500 & 1.4 & 35 & 38 & 0.00251 & 7\\
\hline
\end{tabular}
\end{table}

Table \ref{tab:finite-bounds} uses $\eta=0.002$ and $c=e^{-1}$ throughout.
For exact SVP, $R=2000^{1/n}$; for approximate SVP, $h=1/4$ and $\rho=\gamma/1000^{1/n}$.
The integer $a=\lceil-\log_2q\rceil$ rounds the success lower bound downward.
The stated error and budget give $N(q-\epsilon)\geq8-1/16$, so \eqref{eq:finite-query-success} exceeds $0.999$.
For example, $n=500$ and $\gamma=1.2$ give joint success greater than $0.996$ under the table's accuracy assumption.
The finite overhead over $nK_\gamma^*/\log2$ pays for the prescribed exceptional probability, conservative moment bounds, and amplification.
For the converse column take $\eta=0.002$, $\kappa=1000^{-1/n}$, and $R=2000^{1/n}$ in Theorem \ref{thm:finite-converse}, and put $b=\lfloor-\log_2U\rfloor$.
Its geometric hypotheses hold for every listed pair.
For exact queries, any adaptive procedure with joint success at least $0.99$ must have $N>0.98649\,2^b$ by \eqref{eq:finite-query-converse}.
The row with $b=0$ gives only a trivial lower bound.
For $n=500$ and $\gamma=1.2$, the certificates thus give a necessary budget above $0.98649\,2^{47}$ and a sufficient budget $2^{82}$, with their respective stated accuracy conditions.
The finite gap is retained explicitly; the two bounds agree only on the exponential scale as $n\to\infty$.

\section{Conclusions and further questions}
\label{sec:conclusion}

Direct Gaussian selection has a sharp sample exponent on Haar random lattices.
Choosing the width from the full history and reducing each sampled vector to its primitive part do not lower this exponent.
The balanced-width comparison of ADRS explains the unmodified selection problem.
The explicit Haar rate and its preservation under primitive reduction quantify the limitation for the larger output model.
The finite converse separates direct hits from recovery through multiples, and the finite attainment bound counts the same Gaussian queries.

Generation costs remain a separate design question.
The cost formulas distinguish improvement over the Pouly--Shen direct-selection implementation from competition with later algorithms that combine Gaussian information.
Further questions include the subexponential factors at the threshold, samplers whose amortized costs meet these benchmarks, and other forms of preprocessing that preserve or change the direct-selection bound.
Extending the result to another lattice ensemble requires estimates for the successful numerator and the full Gaussian denominator on a common typical set.

\section*{Data availability}
The values in Tables \ref{tab:search-costs} and \ref{tab:finite-bounds} and Figure \ref{fig:draw-rates} are evaluations of the displayed formulas.
No lattice sample data are used.
The source includes reproduction routines for both sides of the finite query table.

\section*{Acknowledgments}

This work was supported by JSPS KAKENHI Grant Number 25K07752.

We used ChatGPT (OpenAI) as a research and writing aid during manuscript revision.
In particular, it assisted in improving the exposition, formulating and
drafting an initial version of the finite-dimensional converse argument in Section~\ref{sec:finite-converse} 
and the associated query bound, and generating code for numerical evaluation.
The author checked and revised all mathematical arguments and verified the numerical computations, 
and takes full responsibility for the content of the manuscript.

\label{paper:body-end}


\begin{thebibliography}{99}

\bibitem{Sodergren11}
A. S\"odergren,
On the Poisson distribution of lengths of lattice vectors in a random
lattice,
Mathematische Zeitschrift 269 (2011), no. 3--4, 945--954.
\url{https://doi.org/10.1007/s00209-010-0772-8}

\bibitem{Laarhoven26Spherical}
T. Laarhoven,
Spherical statistics and phase transitions in high-dimensional Random lattices,
arXiv:2609.13635, version 1, 2026.
\url{https://arxiv.org/abs/2609.13635}

\bibitem{Ajtai96}
M. Ajtai,
Generating hard instances of lattice problems,
in Proceedings of the 28th ACM Symposium on Theory of Computing,
ACM, New York, 1996, pp. 99--108.
\url{https://doi.org/10.1145/237814.237838}

\bibitem{MicciancioGoldwasser02}
D. Micciancio and S. Goldwasser,
Complexity of Lattice Problems: A Cryptographic Perspective,
The Kluwer International Series in Engineering and Computer Science, vol. 671,
Kluwer Academic Publishers, Boston, MA, 2002.
\url{https://doi.org/10.1007/978-1-4615-0897-7}

\bibitem{MicciancioRegev07}
D. Micciancio and O. Regev,
Worst--case to average--case reductions based on Gaussian measures,
SIAM Journal on Computing 37 (2007), no. 1, 267--302.
\url{https://doi.org/10.1137/S0097539705447360}

\bibitem{Regev09}
O. Regev,
On lattices, learning with errors, random linear codes, and cryptography,
Journal of the ACM 56 (2009), no. 6, Article 34, 40 pp.
\url{https://doi.org/10.1145/1568318.1568324}

\bibitem{Peikert16}
C. Peikert,
A decade of lattice cryptography,
Foundations and Trends in Theoretical Computer Science 10 (2016), no. 4,
283--424.
\url{https://doi.org/10.1561/0400000074}

\bibitem{AjtaiKumarSivakumar01}
M. Ajtai, R. Kumar, and D. Sivakumar,
A sieve algorithm for the shortest lattice vector problem,
in Proceedings of the 33rd ACM Symposium on Theory of Computing,
ACM, New York, 2001, pp. 601--610.
\url{https://doi.org/10.1145/380752.380857}

% After acceptance, replace the arXiv-only entry below by the Physica Scripta bibliographic data.
\bibitem{KaminagaThermal}
M. Kaminaga,
Thermal concentration and Poisson--Dirichlet edge statistics for
random--lattice Gibbs ensembles,
arXiv:2607.00311, version 2, 2026.
\url{https://arxiv.org/abs/2607.00311v2}


\bibitem{Rogers55}
C. A. Rogers,
Mean values over the space of lattices,
Acta Mathematica 94 (1955), 249--287.
\url{https://doi.org/10.1007/BF02392493}

\bibitem{ADRS15}
D. Aggarwal, D. Dadush, O. Regev, and N. Stephens-Davidowitz,
Solving the shortest vector problem in $2^n$ time using discrete Gaussian sampling: Extended abstract,
in Proceedings of the 47th ACM Symposium on Theory of Computing,
2015, pp. 733--742.
\url{https://doi.org/10.1145/2746539.2746606}
Full version: arXiv:1412.7994v5.
\url{https://arxiv.org/abs/1412.7994v5}

\bibitem{PoulyShen26}
A. Pouly and Y. Shen,
Solving the shortest vector problem in $2^{0.63269n+o(n)}$ time on random lattices,
in Advances in Cryptology, EUROCRYPT 2026, Part IV,
Lecture Notes in Computer Science 16544, Springer, 2026, pp. 92--123.
\url{https://doi.org/10.1007/978-3-032-25327-9_4}
Full version: Cryptology ePrint Archive, Paper 2024/1805.
\url{https://eprint.iacr.org/2024/1805}
Authors' EUROCRYPT 2026 slides:
\url{https://www.pouly.fr/data/slides/eurocrypt2026.pdf}

\bibitem{Kim26DGS}
J. Kim,
One discrete Gaussian sample in $2^{n/2+o(n)}$ time,
arXiv:2608.03220, version 1, 2026.
\url{https://arxiv.org/abs/2608.03220v1}

\bibitem{Hhan26Coset}
M. Hhan,
Finding a shortest vector and more in $2^{n/2+o(n)}$ time using $q$-ary coset difference tree,
arXiv:2609.02764, version 1, 2026.
\url{https://arxiv.org/abs/2609.02764v1}
Cryptology ePrint Archive, Paper 2026/1859.
\url{https://eprint.iacr.org/2026/1859}

\bibitem{SD16}
N. Stephens-Davidowitz,
Discrete Gaussian sampling reduces to CVP and SVP,
in Proceedings of the 27th Annual ACM--SIAM Symposium on Discrete Algorithms,
2016, pp. 1748--1764.
\url{https://doi.org/10.1137/1.9781611974331.ch121}

\end{thebibliography}
\end{document}